\documentclass{amsart}

\usepackage{amssymb}
\usepackage{amsfonts}
\usepackage{amsthm}
\usepackage{epsfig,amsmath}
\usepackage{graphicx}
\usepackage{caption}
\usepackage{subcaption}
\usepackage{wrapfig}

\usepackage[margin=1.00 in]{geometry}

\makeatletter
\def\ps@pprintTitle{%
 \let\@oddhead\@empty
 \let\@evenhead\@empty
 \def\@oddfoot{}%
 \let\@evenfoot\@oddfoot}
\makeatother

\newtheorem{theorem}{Theorem}[section]
\newtheorem{definition}{Definition}
\newtheorem{lemma}[theorem]{Lemma}

\newtheorem{conjecture}[theorem]{Conjecture}

\begin{document}

\title{Volume and determinant densities of hyperbolic rational links}

\begin{abstract}
The \emph{volume density} of a hyperbolic link is defined as the ratio of hyperbolic volume to crossing number. We study its properties and a closely-related invariant called the \emph{determinant density}. It is known that the sets of volume densities and determinant densities of links are dense in the interval $[0,v_{oct}]$. We construct sequences of alternating knots whose volume and determinant densities both converge to any $x \in [0,v_{oct}]$. We also investigate the distributions of volume and determinant densities for hyperbolic rational links, and establish upper bounds and density results for these invariants.
\end{abstract}

\author[C. Adams]{Colin Adams}
\address[Colin Adams]{Williams College}
\email{colin.c.adams@williams.edu}

\author[A. Calderon]{Aaron Calderon}
\address[Aaron Calderon]{University of Nebraska-Lincoln}
\email{aaron.calderon@huskers.unl.edu}

\author[X. Jiang]{Xinyi Jiang}
\address[Xinyi Jiang]{Stanford University}
\email{xinyij@stanford.edu}

\author[A. Kastner]{Alexander Kastner}
\address[Alexander Kastner]{Williams College}
\email{ask2@williams.edu}

\author[G. Kehne]{Gregory Kehne}
\address[Gregory Kehne]{Williams College}
\email{gtk1@williams.edu}

\author[N. Mayer]{Nathaniel Mayer}
\address[Nathaniel Mayer]{Harvard University}
\email{nmayer26@gmail.com}

\author[M. Smith]{Mia Smith}
\address[Mia Smith]{Williams College}
\email{mia.smith25@gmail.com}

\date{\today}

\maketitle

\section{Introduction}

There is a deep connection between the hyperbolic volume of links and the combinatorics of their diagrams, which is not fully understood. In \cite{Champanerkar}, Champanerkar, Kofman and Purcell studied the relationship between the volumes and the determinants of links. We continue this investigation, focusing on the special case of rational (2-bridge) links.

The \emph{volume density} $\mathcal{D}_{vol}$ of a hyperbolic link $L$  is the ratio of volume to crossing number:
\[\mathcal{D}_{vol}(L) := \frac{vol(L)}{c(L)}.\]
D. Thurston showed that the complement of a link $L$ in $S^3$ can be decomposed by placing an octahedron at each crossing, as explained in \cite[Section 5]{Adams2012}. Since the volume of a hyperbolic octahedron is bounded above by $v_{oct} = 3.6638 \dots$, the volume of an ideal regular octahedron, we have the upper bound
\[\mathcal{D}_{vol}(L) \le v_{oct}.\]
This decomposition was modified slightly in \cite{Adams2012} to yield $vol(L) \le v_{oct} (c(L)-5) + 4 v_{tet}$ for any hyperbolic link $L$ with $c(L) \ge 5$, where $v_{tet} = 1.0149 \dots$ is the volume of an ideal regular tetrahedron. Hence, the volume density of any link is strictly less than $v_{oct}$.

The authors of \cite{Champanerkar} defined the \emph{determinant density} for a link $L$, as follows:
\[\mathcal{D}_{\det}(L) := \frac{2 \pi \log \det(L)}{c(L)},\]
where the \emph{determinant} of a link $L$ is given by $\det(L) := |\Delta_L(-1)|$ and $\Delta_L$ denotes the Alexander polynomial of $L$. For alternating links, the determinant is equal to the number of spanning trees of either Tait graph (see \cite{Crowell1959}). 
Data produced by Dunfield \cite{Dunfield} suggests that the volume and determinant densities are closely related. The following conjecture appears in \cite{Champanerkar} and is equivalent to a conjecture by Kenyon about planar graphs \cite{Kenyon}:

\begin{conjecture}
\label{Determinant density UB}
For any link $L$,
\[\mathcal{D}_{\det}(L) \le v_{oct}.\]
\end{conjecture}

In \cite{Champanerkar}, Champanerkar, Kofman and Purcell constructed \emph{geometrically} and \emph{diagrammatically maximal} sequences of links related to the \emph{infinite square weave} (shown in Figure \ref{InfWeave}), that is, sequences whose volume and determinant densities both approach $v_{oct}$. Thus, if the conjecture holds, $v_{oct}$ is an asymptotically sharp upper bound for both densities. Further, they proposed the following simple relationship between the two quantities, which they have verified for all alternating knots of up to 16 crossings.

\begin{figure}[h]
\centering
\includegraphics[scale=0.3]{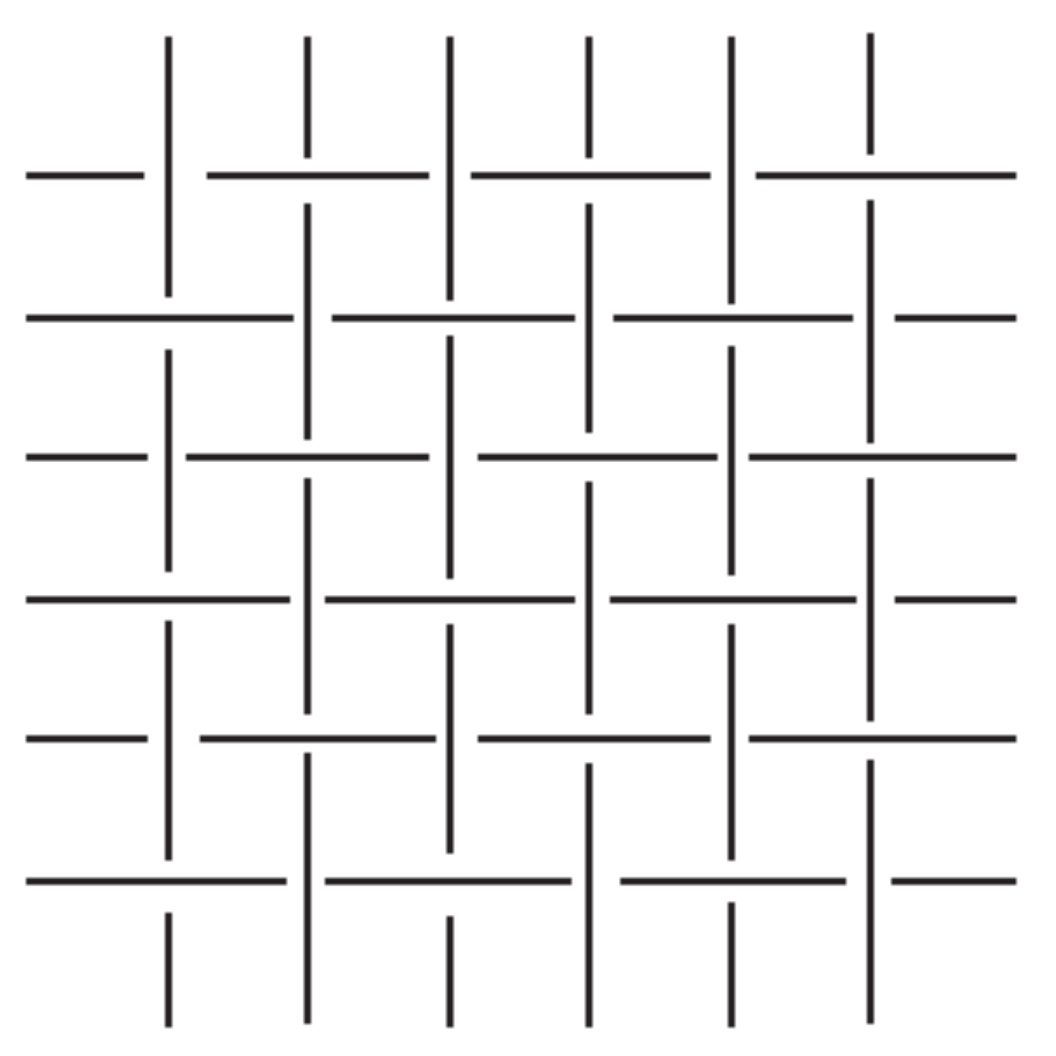}
\caption{The infinite square weave}
\label{InfWeave}
\end{figure}

\begin{conjecture}
\label{Volume-determinant conjecture}
For any alternating hyperbolic link $L$,
\[vol(L) < 2 \pi \log \det(L).\]
\end{conjecture}

Some progress has been made towards this conjecture in the general case. For instance, Stoimenow showed in \cite{Stoimenow} that for any hyperbolic alternating link $L$, $2 \cdot 1.0355^{vol(L)} \le \det(K)$, i.e.
\[4.355 + 0.219 \cdot vol(L) < 2 \pi \log \det(L).\]

\begin{definition}
\normalfont
Let $\mathcal{C}_{vol}$ and $\mathcal{C}_{det}$ denote the respective sets of volume and determinant densities for hyperbolic links.
\end{definition}

Much recent work has focused on understanding the structure of $\mathcal{C}_{vol}$ and $\mathcal{C}_{\det}$. In particular, Burton \cite{Burton} showed that $\mathcal{C}_{vol}$ and $\mathcal{C}_{\det}$ are dense in the interval $[0,v_{oct}]$. In Section 2, we extend these results to prove that the set of volume and determinant densities of hyperbolic knots (i.e. one-component links) are dense in $[0,v_{oct}]$.  Moreover, we show that for any $x \in [0,v_{oct}]$ there exists a sequence of knots $\{K_n\}$ such that $\mathcal{D}_{vol}(K_n)$ and $\mathcal{D}_{det}(K_n)$ both approach $x$. 

We note in passing that the determinant of alternating knots and links is related to other quantum invariants such as the average of the absolute values of the coefficients of the Jones polynomial, the rank of the reduced Khovanov homology, and the Kashaev invariant. Hence, the density results for determinant densities imply similar density results for these other quantities (see \cite{ChampanerkarSpectra} for more details).

The main purpose of this paper is to study the distributions of volume and determinant densities for the class of hyperbolic rational links. Every rational link is alternating, and can be isotoped to a \emph{standard alternating 2-bridge representation} with one free strand on the right, as shown in Figure \ref{R(4,1,1,2)}. We call the ``vertical chains of bigons" the \emph{twist levels} of the representation. We use the expression $[a_1,a_2,\dots,a_n]$, $a_i > 0$, to denote the rational link with a standard alternating 2-bridge representation with $n$ twist levels, where the $i$th twist level has $a_i$ crossings.

\begin{figure}[h]
\centering
\includegraphics[scale=0.3]{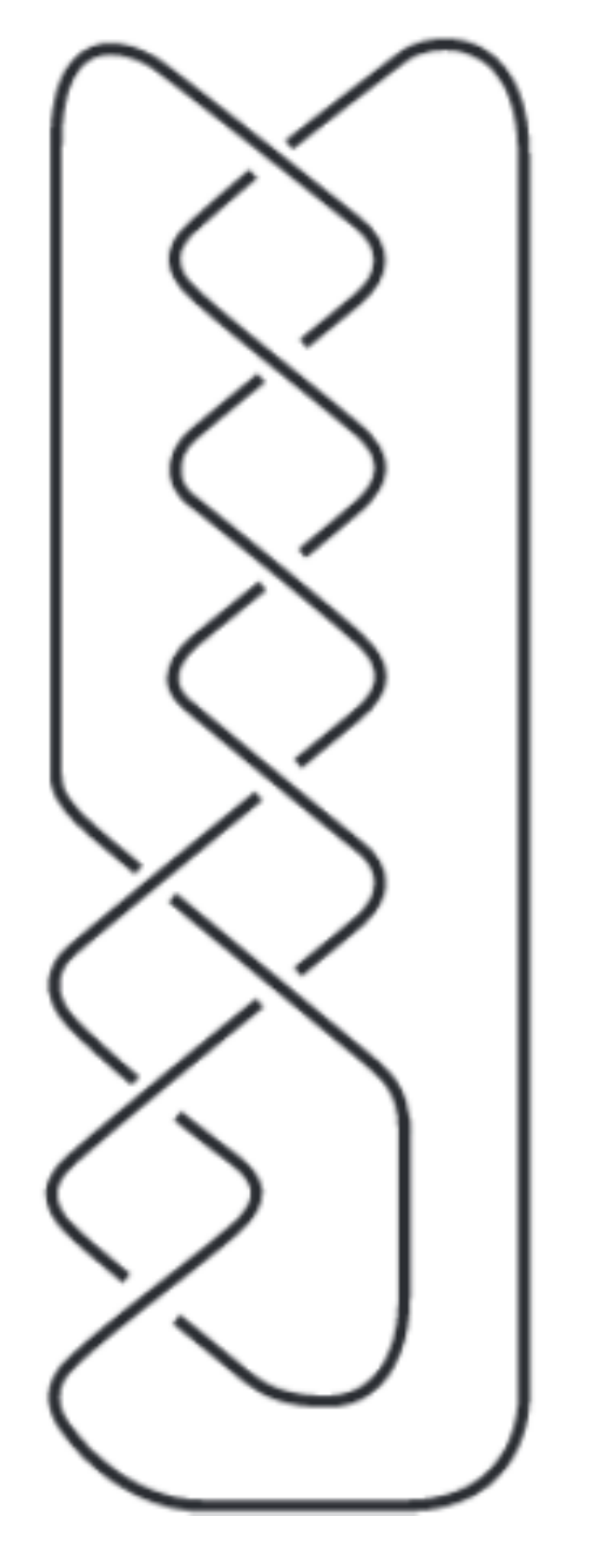}
\caption{The standard alternating 2-bridge representation of $[4,1,1,2]$.}
\label{R(4,1,1,2)}
\end{figure}

\begin{definition}
\normalfont
Let $\mathcal{C}_{vol}^{rat}$ and $\mathcal{C}_{det}^{rat}$ denote the sets of volume and determinant densities, respectively, for hyperbolic rational links.
\end{definition}

We prove the following theorems in Sections 3 and 4:

\newtheorem*{volume dense and bounded}{Theorem \ref{volume dense and bounded}}
\begin{volume dense and bounded}
The set
$\mathcal{C}_{vol}^{rat}$ is a dense subset of $[0, 2 v_{tet}]$.
\end{volume dense and bounded}

\newtheorem*{det dense and bounded}{Theorem \ref{det dense and bounded}}
\begin{det dense and bounded}
The set $\mathcal{C}_{det}^{rat}$ is a dense subset of $[0, 2 \pi \log(\phi)]$, where $\phi = \frac{1+\sqrt{5}}{2}$ is the golden ratio.
\end{det dense and bounded}


Note that the least upper bounds for $\mathcal{C}_{vol}^{rat}$ and $\mathcal{C}_{det}^{rat}$ differ, and are both smaller than $v_{oct}$. In \cite{Champanerkar}, the links used to form geometrically and diagrammatically maximal sequences contain increasingly large patches of the infinite square weave. This cannot be achieved in the case of rational links, so it stands to reason that the corresponding least upper bounds might be smaller. It is an interesting problem to attempt similar results for 3-bridge hyperbolic links and for $k$-bridge hyperbolic links in general. As $k$ increases we can construct links which contain incrementally larger portions of the infinite square weave, and therefore we expect the least upper bounds to increase as well. This motivates the following conjecture, which would provide a characterization of the structure of $\mathcal{C}_{vol}$ and $\mathcal{C}_{det}$.

\begin{conjecture}
Let $\mathcal{C}_{vol}^k$ and $\mathcal{C}_{det}^k$ represent the set of volume and determinant densities of $k$-bridge links, respectively, and let $\beta_{vol}^k$ and $\beta_{det}^k$ denote their respective least upper bounds. Then $\{\beta_{vol}^k\}_{k \in \mathbb{N}}$ and $\{\beta_{det}^k\}_{k \in \mathbb{N}}$ are strictly increasing sequences that both converge to $v_{oct}$.
\end{conjecture}

\section{Density results for hyperbolic knots}


\begin{figure}
    \label{potholder}
    \centering
    \begin{subfigure}[b]{0.3\textwidth}
        \includegraphics[scale = 0.12]{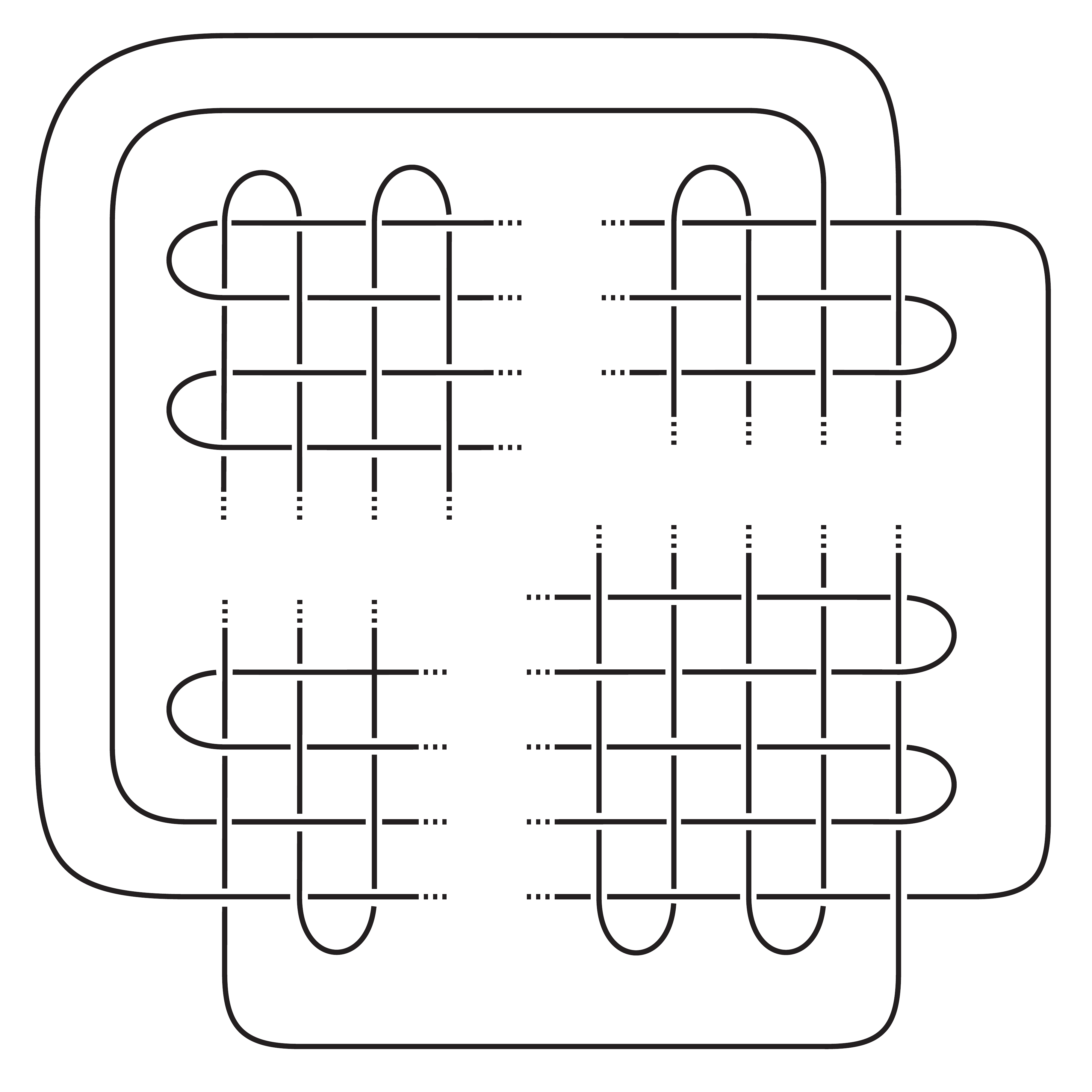}
        \caption{$W_n$ for even $n$}
        \label{Weven}
    \end{subfigure}
    \begin{subfigure}[b]{0.3\textwidth}
        \includegraphics[scale=0.14]{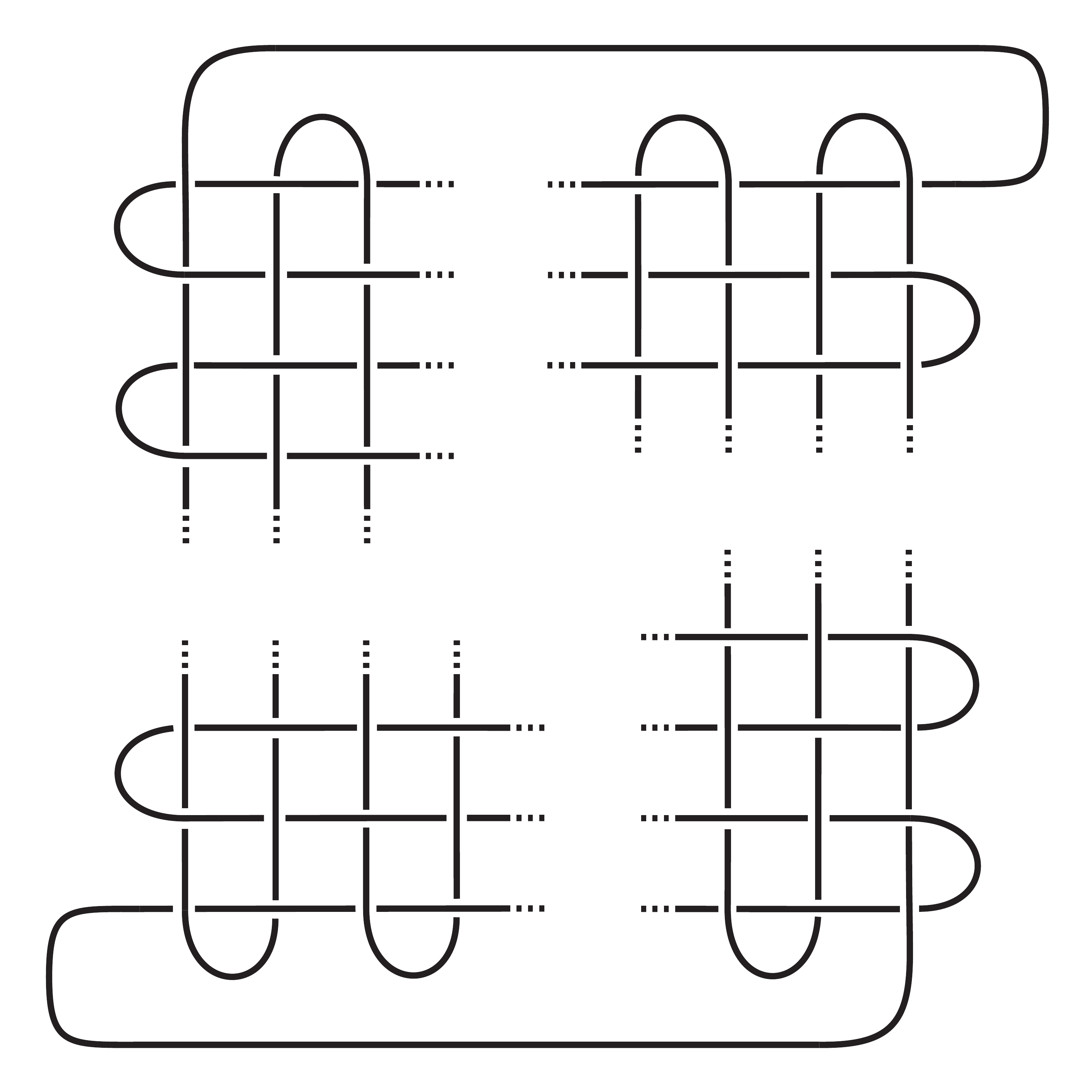}
        \caption{$W_n$ for odd $n$}
        \label{Wodd}
    \end{subfigure}
    \caption{The `modified potholder closure' for $W_n$. Note that for all $n$, $W_n$ is an alternating knot with $n^2$ crossings.}
 \end{figure}

\begin{theorem}
Given $x \in [0,v_{oct}]$, there exists a sequence of hyperbolic knots whose volume and determinant densities both converge to $x$.
\end{theorem}

\begin{proof}
Let $W_n$ be the `modified potholder closure' of the $n$-by-$n$ weave tangle (Figures \ref{Weven}, \ref{Wodd}), and let $W_n^m$ be the result of  lengthening a twist region on the boundary of the projection of $W_n$ by $m$ crossings, where $m$ is even. It will not matter which twist region we pick. By this construction, $W_n$ is an alternating knot with $n^2$ crossings and $W_n^m$ is an alternating knot with $n^2 +m$ crossings. By \cite{Menasco}, $W_n$ and $W_n^m$ are hyperbolic, because they have alternating non-two braid prime diagrams. In \cite{Champanerkar} the sequence $\{W_n\}$ was shown to be geometrically and diagrammatically maximal:
\[\lim_{n \to \infty}\mathcal{D}_{vol}(W_n)=\lim_{n \to \infty}\mathcal{D}_{det}(W_n)=v_{oct}.\]

We first consider the behavior of $\mathcal{D}_{vol}(W_n^m)$ as $m$ and $n$ become large. Let $B(W_n)$ denote the link obtained by replacing the ``additional crossings'' in the twist region of $W_n^m$ with a belting component.  This means that $m$ of the crossings in the twist region are undone, and a trivial component is added around what used to be the twist sequence (similar to Figure \ref{R4B}). Since $\mathcal{D}_{vol}$ is bounded above by $v_{oct}$,  and $B(W_n)$ has $n^2 + 4$ crossings, $n^2$ coming from $W_n$ and 4 coming from the belt, we have that $vol(B(W_n))\leq (n^2+4)v_{oct}$. We also know that $(1,k)$-Dehn filling decreases volume \cite{ThurstonNotes}, and therefore $vol(W_n^m)<vol(B(W_n))$. Finally, a result of Futer, Kalfagianni, and Purcell [Futer2006] provides a bound for the amount the volume can drop due to Dehn filling along a slope $s$, where $\ell$ is the length of $s$, and $\ell > 2 \pi$:
\[\left( 1 - \left( \frac{2\pi}{\ell} \right)^2 \right) ^{3/2} vol(B(W_n)) \leq vol(W_n^m).\]
Note that $\ell$ increases without bound as $m$ does. Combining these results,
\[\left(1 - \left( \frac{2\pi}{\ell} \right)^2 \right) ^{3/2} vol(W_n) < \left( 1 - \left( \frac{2\pi}{\ell} \right)^2 \right) ^{3/2} vol(B(W_n)) \leq vol(W_n^m) < vol(B(W_n)) \leq (n^2+4) v_{oct}.\]
Dividing by $n^2+m$ and taking the limit as $m$ and $n$ (and therefore $\ell$) approach infinity yields
\[\lim_{m,n \to \infty}\mathcal{D}_{vol}(W_n^m) =\frac{n^2 v_{oct}}{n^2+m}\]
By choosing 
\[ m=2\left \lfloor \frac{1}{2}\left( \frac{n^2 v_{oct}}{x} - n^2 \right) \right \rfloor\]
and letting $n$ increase, we obtain a sequence of hyperbolic knots whose volume densities approach any $x \in [0, v_{oct}]$.

Let us now consider the behavior of $\mathcal{D}_{det}(W_n^m)$. The sequence $\{W_n\}$ is both geometrically and diagrammatically maximal, and so for sufficiently large $n$, $\mathcal{D}_{vol}(W_n)$ and $\mathcal{D}_{det}(W_n)$ become  arbitrarily close. Since $W_n^m$ is alternating, $\det (W_n^m)$ is the number of spanning trees of either Tait graph of $W_n^m$. Choose the Tait graph of $W_n^m$ such that the $m+1$ crossings in a single twist region correspond to $m+1$ multi-edges in place of the original edge $e$ in the Tait graph of $W_n$. Then the number of spanning trees $T(W_n)$ is
\[\det(W_n)=T(W_n)= T_e(W_n) +T_{\bar{e}}(W_n),\]
where $T_e(W_n)$ denotes the number of spanning trees of $T(W_n)$ which contain $e$ and $T_{\bar{e}}(W_n)$ denotes the number of spanning trees which do not contain $e$. Once these crossings are added, there are exactly $m+1$ distinct spanning trees for the Tait graph of $W_n^m$ for every spanning tree of the Tait graph of $W_n$ that contains $e$, one for each of the multi-edges. The spanning trees that do not contain $e$ remain unchanged. This yields
\[\det(W_n^m) = T(W_n^m)= (m+1)T_e(W_n) +T_{\bar{e}}(W_n)\]
and therefore
\[\det(W_n) \leq \det(W_n^m) \leq (m+1) \det(W_n).\]
This gives
\[\mathcal{D}_{det}(W_n)\frac{n^2}{n^2+m} \leq \mathcal{D}_{det}(W_n^m) \leq \mathcal{D}_{det}(W_n)\frac{n^2}{n^2+m} +\frac{2\pi\log (m+1)}{n^2+m}.\]
As $m$ and $n$ become large this last term goes to $0$ and $\mathcal{D}_{det}(W_n)$ approaches $v_{oct}$. Therefore 
\[\lim_{m,n \to \infty}\mathcal{D}_{det}(W_n^m) =\frac{n^2 v_{oct}}{n^2+m}\]
as well. Using the same choices of $m$ and $n$ (and therefore the same sequence of knots) as in the above proof for volume density yields
\[\lim_{m,n \to \infty}\mathcal{D}_{det}(W_n^m)=x.\]
\end{proof}

\section{Volume densities of rational links}

In this section, we prove that $\mathcal{C}_{vol}^{rat}$ is a dense subset of $[0,2 v_{tet}]$. Most of the proofs refer to a specific sequence of rational links with volume densities approaching $2v_{tet}$, which we define below.

\begin{definition}
\normalfont
Let $R_n$ be the rational link with the alternating standard 2-bridge representation that has exactly one crossing in each of its $n$ twist levels (Figure \ref{R4}). Let $R_n^m$ be the result of adding $m$ crossings to the last twist level of $R_n$ (Figure \ref{R45}), $m$ an even number. Finally, let $B(R_n)$ denote $R_n$ with a belt about its last twist level (Figure \ref{R4B}).
\end{definition}

\begin{figure}
\label{R4variants}
    \centering
    \begin{subfigure}[b]{0.3\textwidth}
        \includegraphics[scale = 0.32]{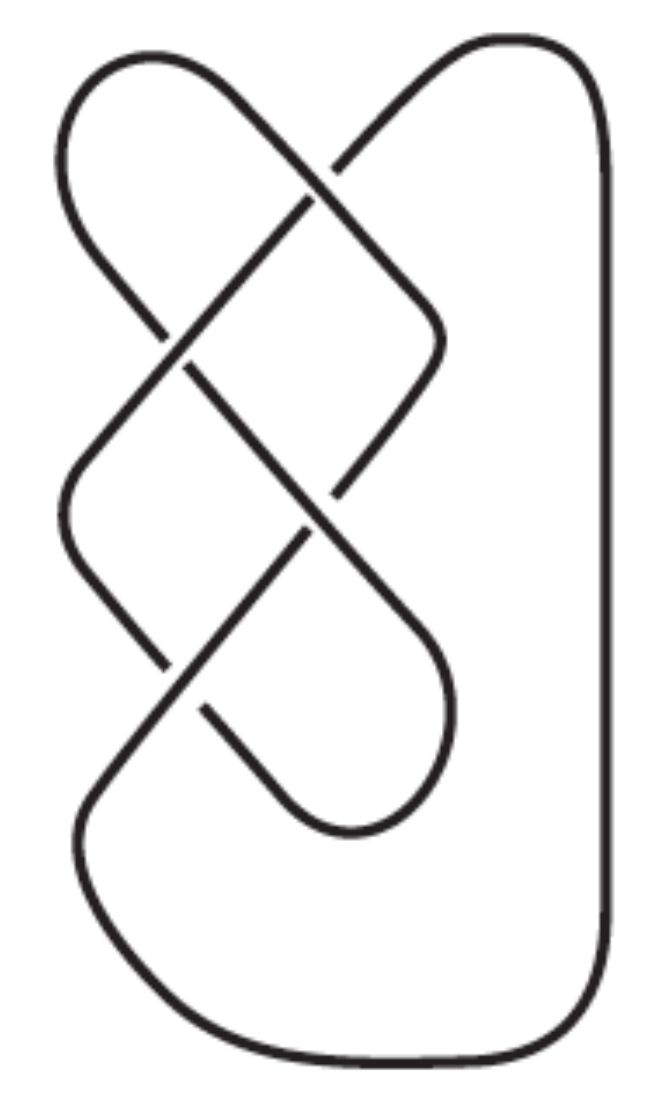}
        \caption{$R_4$}
        \label{R4}
    \end{subfigure}
    \begin{subfigure}[b]{0.3\textwidth}
        \includegraphics[scale=0.25]{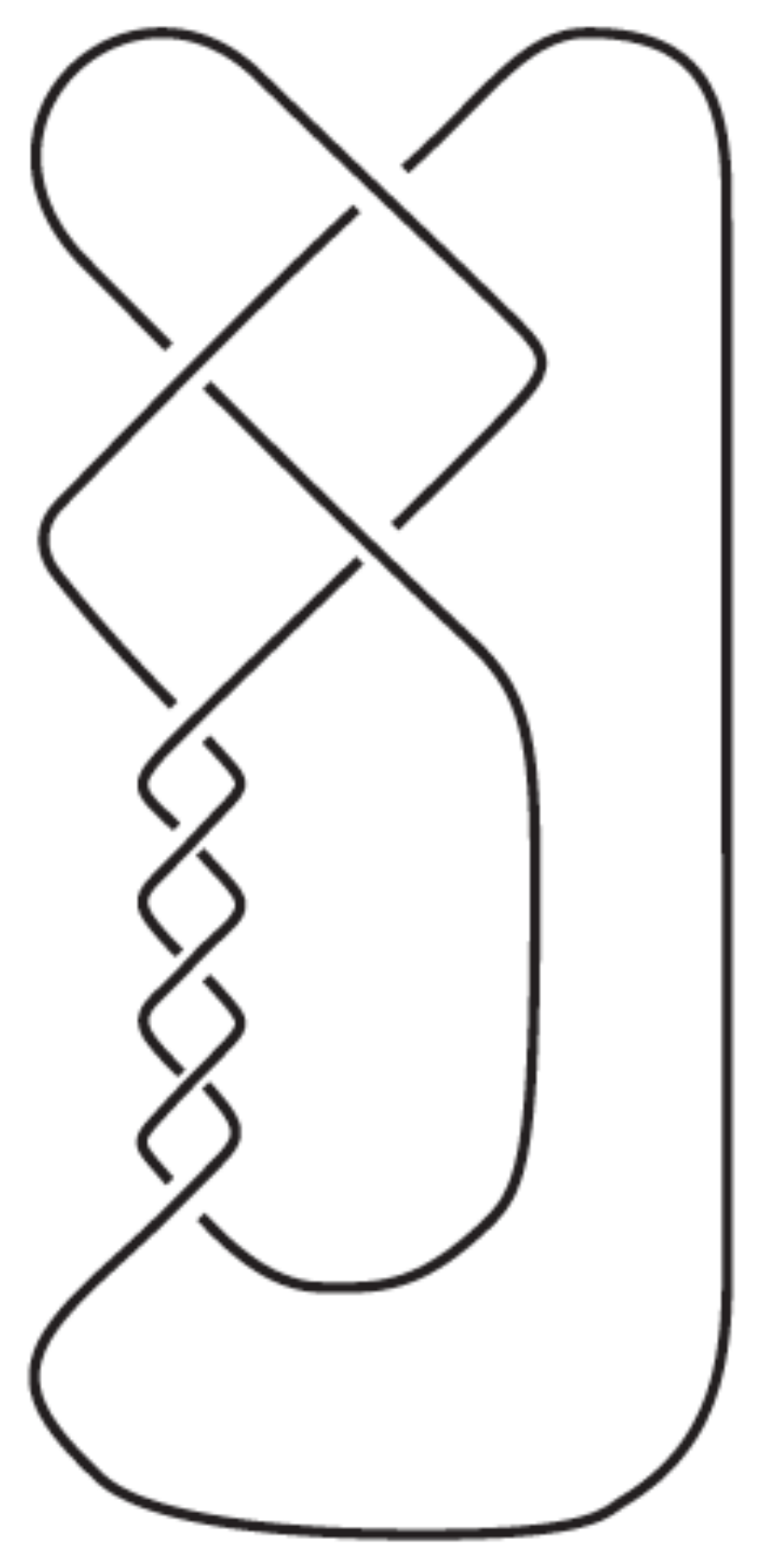}
        \caption{$R_4^5$}
        \label{R45}
    \end{subfigure}
        \begin{subfigure}[b]{0.3\textwidth}
        \includegraphics[scale = 0.5]{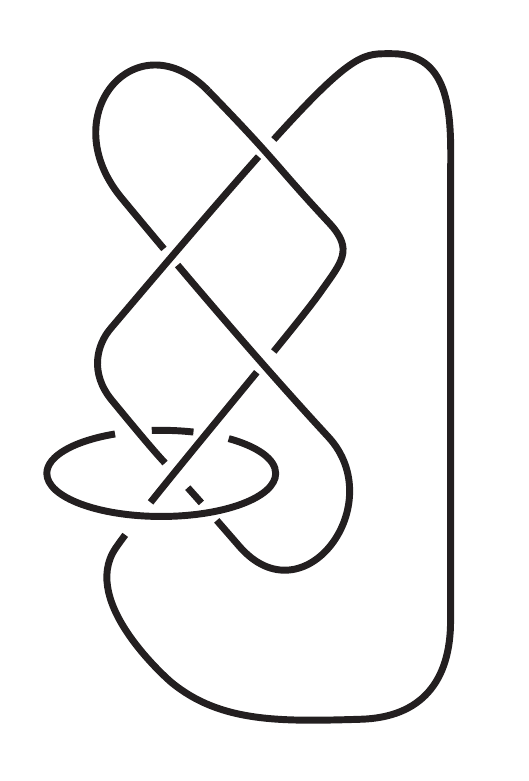}
        \caption{$B(R_4)$}
        \label{R4B}
    \end{subfigure}
    \caption{Variants on $R_4$.} 
\end{figure}

In order to prove density results specifically for rational knots, we will make use of the following proposition.

\begin{lemma}
\label{whenknot}
If $n \equiv 0,1\ (\textrm{mod}\ 3)$ and $m$ is even, then $R_n^m$ is a knot. 
\end{lemma}
\begin{proof}
Observe that in $R_n$, every set of $6$ crossings returns the three strands to their original positions. Hence, we only need to consider $n$ modulo $6$. Checking each of the $6$ states, only $n \equiv 0,1\ (\textrm{mod}\ 3)$ yield knots. Moreover, any odd number of crossings in the last twist level preserves the number of components. 
\end{proof}

\begin{lemma}
\label{maxseq}
$R_n$ is hyperbolic if $n \ge 4$, and
\[\lim_{n \to \infty} \mathcal{D}_{vol}(R_n) = 2v_{tet}.\]
\end{lemma}

\begin{proof}
If $n \ge 4$, then the standard 2-bridge representation of $R_n$ is reduced alternating, connected, with no circles intersecting the projection in two points with crossings on both sides and not obviously a 2-braid. Theorems of Menasco imply that $R_n$ is in fact prime, non-split and not 2-braid \cite{Menasco198437}. Hence, $R_n$ is hyperbolic when $n \ge 4$.

To obtain a lower bound for $vol(R_n)$, we use the following inequality due to Futer \cite[Theorem B.3]{Gueritaud}:
\[ 2v_{tet} tw(R)-2.7066<vol(R),\]
where $R$ is any hyperbolic rational link and $tw(R)$ is its number of twist regions.
Following the methods of \cite{Adams2015}, an upper bound for $vol(R_n)$ can be found by using face-centered bipyramids and collapsing the largest two faces in the projection. This yields
\[vol(R_n)\leq 2v_{tet}(n-2).\]
For all $n \ge 4$, the crossing number is $c(R_n)=n$, and therefore
\[ \frac{2(n-2)v_{tet}-2.7066}{n}< \mathcal{D}_{vol}(R_n) \leq \frac{2(n-2)v_{tet}}{n}.\]
Thus
\[ \lim_{n \to \infty} \mathcal{D}_{vol}(R_n) =2v_{tet}.\]

Restricting to the links $R_n$ that satisfy Lemma \ref{whenknot} yields a sequence of hyperbolic rational knots whose volume densities converge to $2 v_{tet}$ from below.
\end{proof}

\begin{lemma}
\label{volume bounded}
For all rational links $L$, $\mathcal{D}_{vol}(L) < 2v_{tet}$.
\end{lemma}

\begin{proof}
Suppose that $L$ is a rational link with $\mathcal{D}_{vol}(L) \ge 2v_{tet}$, and suppose that it has a standard 2-bridge projection with $n$ twist levels. In the face-centered bipyramid upper bound on volume appearing in \cite{Adams2015},the bipyramids corresponding to the exterior face and the face that shares with it the vertical strand with no crossings can be collapsed.  Each remaining face $F_i$ of the projection is associated with exactly one twist level. Therefore, if $e_i$ is the number of edges of the face $F_i$ and $a_i$ is the number of crossings in the $i^{th}$ twist level, then
\[ e_i=\begin{cases} 
      a_i+1 & i=\{1,n\} \\
      a_i+2 & i \in \{2, \ldots, n-1\}
      \end{cases}
\]
Since the face-centered bipyramid method gives an upper bound $\beta(L)$ on the volume of $L$, if $\mathcal{D}_{vol}(L) \ge 2v_{tet}$ then $\frac{\beta(L)}{c(L)} \ge 2v_{tet}$ as well. However, if a regular ideal $k$-sided bipyramid $B_k$ is associated with each face of $k$ edges, then each of the $n$ twist levels contribute $vol(B_{e_i})$ to $\beta(L)$. This gives
\[\sum_{i=1}^n vol(B_{e_i}) \ge 2 v_{tet} \sum_{i=1}^n a_i.\]
Therefore either there exists an $i$ such that $\frac{vol(B_{e_i})}{a_i} > 2v_{tet}$, or it is the case that for all $i$, $ \frac{vol(B_{e_i})}{a_i} = 2 v_{tet}$. 

Consider the first case. By the fact that the sequence $\{vol(B_k)\}$ is strictly increasing \cite{Adams2015}, we need only consider $i$ such that 
\[ \frac{vol(B_{a_i+2})}{a_i} > 2v_{tet}.\]
It was proved in \cite{Adams2015} that for $n\geq 2$
\[vol(B_n)< 2 \pi \text{log}(n).\]
Note that for $a_i=7$, $vol(B_{7+2})<2\pi \text{log}(7 + 2)<14v_{tet}$, and that for all $a_i\geq 8$, 
\[ \frac{d}{da_i} 2\pi\text{log}(a_i+2)< \frac{d}{da_i} 2a_iv_{tet}.\] 
This implies that for such an $i$, $a_i\leq 6$. However the bipyramid volumes (which appear in \cite{Adams2015}) for these remaining cases are too small, so such an $a_i$ cannot exist.

\bigskip

In the second case, for the first twist region, we have that
\[\frac{vol(B_{a_1+1})}{a_1} = 2 v_{tet}.\]
But the strictly increasing nature of $\{vol(B_k)\}$ together with the result of case one yields
\[vol(B_{a_1+1}) < vol(B_{a_1+2}) \le 2 v_{tet} a_1,\]
which also yields a contradiction.

\end{proof}

\begin{lemma}
\label{Volumes dense}
The set of volume densities of rational knots is dense over the interval $[0,2v_{tet}]$.
\end{lemma}

\begin{proof}

Let $x \in [0,2v_{tet}]$.  Given an $\varepsilon >0$, the goal is to give an $R_n^m$ satisfying Lemma \ref{whenknot} such that
\[ \left|\mathcal{D}_{vol}(R_n^m)-x\right|<\varepsilon.\]

For $x=2v_{tet}$, we have already found such a maximal sequence of rational knots in Lemma \ref{maxseq}. For $x\in [0,2v_{tet})$, choose $m$ and $n$ such that $m\geq 1$, $m$ even, $n \geq 4$, $n \equiv 0,1\ (\textrm{mod}\ 3)$, and 
\[\left|\frac{2nv_{tet}}{n+m-1}-x\right|<\frac{\varepsilon}{2}.\]
From \cite{ThurstonNotes} it is known that
\[\lim_{m \to \infty} vol(R_n^m) = vol(B(R_n))\] 
and that
\[vol(R_n)<vol(B(R_n)).\]
Together these statements imply that for any $n\geq 2$ there exists an $M$ such that for all $m>M$, 
\[vol(R_n)<vol(R_n^m).\]

For such an $m$ we have the following inequality, where the lower bound follows from \cite[Theorem B.3]{Gueritaud} and the upper bound follows from the face-centered bipyramids decomposition of the complement of $B(R_n)$ \cite{Adams2015}:

\[2(n-2)v_{tet}-2.7066< vol(R_n)<vol(R_n^m)<vol(B(R_n))\leq 2(n-2)v_{tet}+2v_{oct}.\] 
Together this yields
\[\left|vol(R_n^m)-2nv_{tet}\right|<4v_{tet}+2.7066<7.\]
Dividing by crossing number, this becomes
\[ \left| \mathcal{D}_{vol}(R_n^m)-\frac{2nv_{tet}}{c(R_n^m)}\right|<\frac{7}{n+m}.\]
Even with the previous constraints on $m$ and $n$, we may choose a sufficiently large $\alpha \in \mathbb{N}$, $\alpha \equiv 0\ (\textrm{mod}\ 3)$ such that
\[\left|\mathcal{D}_{vol}(R_{\alpha n}^{\alpha m})-\frac{2\alpha nv_{tet}}{c(R_{\alpha n}^{\alpha m})}\right|<\frac{7}{\alpha n+\alpha m}<\varepsilon/2.\]
Then setting $M=\alpha m$ and $N= \alpha n$, $M$ and $N$ satisfy Lemma \ref{whenknot} and 
\[\biggl| \mathcal{D}_{vol}(R_{N}^{M})-x \biggr| \leq \left|\mathcal{D}_{vol}(R_{N}^{M})-\frac{2Nv_{tet}}{N+M}\right| + \left|\frac{2Nv_{tet}}{N+M}-x\right|<\varepsilon. \]

\end{proof}

\begin{theorem}
\label{volume dense and bounded} The set
$\mathcal{C}_{vol}^{rat}$ is a dense subset of $[0,2 v_{tet}]$.
\end{theorem}

\begin{proof}
This follows immediately from Lemma \ref{volume bounded} and Lemma \ref{Volumes dense}.
\end{proof}

\section{Determinant densities of rational links}

We now consider the set of determinant densities of rational links. We prove that  $\mathcal{C}_{det}^{rat}$ is a dense subset of $[0,2 \pi \log(\phi)]$ where $\phi$ is the golden ratio. It is interesting to note that unlike  the case of general hyperbolic links, the closure of $\mathcal{C}_{vol}^{rat}$ is a strict subset of the closure of $\mathcal{C}_{det}^{rat}$. 

To begin, we note that any rational link in our specified projection has a series-parallel Tait graph. Then by \cite{Stoimenow2007}, for all rational links $L$, $\mathcal{D}_{det}(L) < 2 \pi \log(\phi)$. We show that this upper bound is the least possible.

\begin{lemma}
\label{det UB}
The determinant densities of the sequence of rational links $R_n$ monotonically increase towards an upper bound of  $2 \pi \log(\phi)$, where $\phi$ is the golden ratio.
\end{lemma}
\begin{proof}
The determinant of $R_{n}$ is the $n$-th Fibonacci number \cite{Kauffman}, which we  denote $F_n$. Noting that the $n$-th Fibonacci number can be written as $\frac{\phi^n - \psi^n}{\sqrt{5}}$, where $\psi = \frac{1-\sqrt{5}}{2} < 1$, we have
\[\lim_{n \to \infty} \frac{2 \pi \log \det(R_n)}{c(R_n)} = \lim_{n \to \infty} \frac{2 \pi}{n} \log (\frac{\phi^{n+1}-\psi^{n+1}}{\sqrt{5}}) = 2 \pi \log \phi.\] 
Moreover, taking the derivative of $\mathcal{D}_{det}(R_{n})$with respect to $n$ shows that the determinant density increases monotonically with $n$. 
\end{proof}

\begin{lemma}
\label{Determinants dense}
The set of determinant densities for rational knots is dense in the interval $[0,2\pi\log\phi]$. 
\end{lemma}
\begin{proof}

Let $x \in [0,2\pi\log\phi]$ and let $R_{n}^m$ be the rational knot as defined previously. Given $\varepsilon>0$, the goal is to find an $R_{n}^m$ such that \[\left| \frac{2\pi\log det(R_{n}^m)}{c(R_{n}^m)}-x \right| < \varepsilon. \]
 By the recursive formula for the number of spanning trees \cite{Kauffman}, $det(R_{n}^m) = (m+1) F_n+F_{n-1}$.  Also, $c(R_{n}^m)= m+n$. In other words, we want to find an $m$ and $n$ such that \[ \left|  \frac{2\pi\log ((m+1) F_{n}+F_{n-1})}{m+n}-x \right| < \varepsilon. \]
Observe that \[ \left| \frac{2\pi}{m+n} \log((m+1) F_n + F_{n-1}) - \frac{2\pi}{m+n} \log((m+1) F_n) \right| = \frac{2\pi}{m+n}\log(1+ \frac{F_{n-1}}{(m+1) F_{n}}) \leq \frac{2\pi}{m+n}\log(2).\]
Thus for sufficiently large $n$ \begin{equation} \left| \frac{2\pi}{m+n} \log((m+1) F_{n} + F_{n-1}) - \frac{2\pi}{m+n} \log((m+1) F_{n}) \right| < \frac{\varepsilon}{3}. \end{equation}
From the explicit formula for the $n$-th Fibonacci number,
\[ \frac{2\pi}{m+n} \log((m+1) F_{n}) = \frac{2\pi}{m+n} (\log(m+1) + \log(\phi^{n}- \psi^{n}) - \log(\sqrt{5})). \]
It is therefore also the case that for sufficiently large $n$ \begin{equation} \left| \frac{2\pi}{m+n} (\log(m+1) + \log(\phi^{n}- \psi^{n}) - \log(\sqrt{5})) - \frac{2\pi}{m+n} (\log(m+1) + \log(\phi^{n})) \right| < \frac{\varepsilon}{3}. \end{equation}
Finally, we can choose $m$ and $n$ according to Lemma \ref{whenknot} such that $n$ satisfies (1) and (2), and \begin{equation}\left| \frac{2\pi}{m+n} (\log(m+1) + n\log(\phi)) - x\right| < \frac{\varepsilon}{3}. \end{equation} 
The triangle inequality with $(1), (2)$, and $(3)$ then yields 
\[\left| \frac{2\pi\log det(R_{n}^{m})}{c(R_{n}^{m})}-x \right| = \left| \frac{2\pi\log ((m+1) F_{n}+F_{n-1})}{m+n}-x \right| < \varepsilon. \]
\end{proof}

\begin{theorem}
\label{det dense and bounded}The set 
$\mathcal{C}_{det}^{rat}$ is a dense subset of $[0, 2 \pi \log(\phi)]$, where $\phi = \frac{1+\sqrt{5}}{2}$ is the golden ratio.
\end{theorem}

\begin{proof}
This follows immediately from Lemma \ref{det UB} and Lemma \ref{Determinants dense}.
\end{proof}

\bibliographystyle{amsalpha}
\bibliography{SMALL}

\end{document}